\documentclass[12pt,a4paper]{article}
\usepackage[T1]{fontenc}
\usepackage[british]{babel}
\usepackage[utf8]{inputenc}
\usepackage{amsmath,amssymb,amsthm}
\usepackage{hyperref}

\newcommand\XX{\mathfrak{X}}
\newcommand\F{\mathrm{F}}

\DeclareMathOperator{\N}{N}
\DeclareMathOperator{\Oo}{O}

\DeclareMathOperator{\C}{C}
\DeclareMathOperator{\Soc}{Soc}

\DeclareMathOperator{\Z}{Z}
\DeclareMathOperator{\Fit}{F}
\DeclareMathOperator{\Ker}{Ker}

\DeclareMathOperator{\Aut}{Aut}
\DeclareMathOperator{\Id}{Id}

\DeclareMathOperator{\Q}{Q}

\DeclareMathOperator{\IYB}{IYB}
\DeclareMathOperator{\YB}{YB}

\newtheorem{theorem}{Theorem}
\newtheorem{corollary}[theorem]{Corollary}
\newtheorem{lemma}[theorem]{Lemma}
\newtheorem{question}[theorem]{Question}
\newtheorem{athm}{Theorem}

\theoremstyle{definition}
\newtheorem{definition}[theorem]{Definition}

\title{On a class of involutive Yang-Baxter groups}

\author{H. Meng\thanks{E-mail: hymeng2009@shu.edu.cn. Department of Mathematics
and Newtouch Center for Mathematics of Shanghai University,
Shanghai 200444, P. R. China. },~A. Ballester-Bolinches\thanks{E-mail: Adolfo.Ballester@uv.es. Departament de Matem\'{a}tiques, Universitat de Val\'{e}ncia, Dr. Moliner 50, 46100 Burjassot, Val\'{e}ncia,
Spain.},~P. Jim\'{e}nez-Seral\thanks{E-mail: paz@unizar.es. Departamento de Matem\'{a}ticas, Universidad de Zaragoza, Pedro Cerbuna, 12, 50009 Zaragoza, Spain.}
}
\date{}
\begin{document}
\maketitle

\begin{abstract}
A group is called an involutive Yang-Baxter group (IYB-group) if it is isomorphic to the permutation group of an involutive, non-degenerate set-theoretic solution of the Yang-Baxter equation. This paper investigates finite soluble groups whose Sylow subgroups have nilpotency class at most two, addressing Ced\'{o} and Okni\'{n}ski's question~\cite{CedoOkninski2025} of whether such groups are IYB-groups. We establish that a finite soluble group with Sylow subgroups of class at most two is an IYB-group if its nilpotent residual is $\Q_8$-free. We also prove that a finite soluble group with Sylow subgroups of class at most two and Sylow $2$-subgroups isomorphic to $Q_8$ is an IYB-group. Our approach depends on a careful study of the arithmetic and normal structure of a finite soluble group which is of independent interest.


\end{abstract}

\section{Introduction}
The Yang-Baxter equation plays a fundamental role in mathematical physics and serves as a cornerstone in the theory of quantum groups. In~\cite{Drinfeld1992}, Drinfeld proposed investigating set-theoretic solutions to this equation. A set-theoretic solution of the Yang–Baxter
equation is a pair $(X,r)$, where $X$ is a nonempty set and $r: X\times X \rightarrow X \times X$ is a map such that
$$r_{12}r_{23}r_{12}=r_{23}r_{12}r_{23},$$
where $r_{12}=r\times \Id_X, r_{23}=\Id_X \times r$ are maps from $X^3$ to $X^3$. Subsequent research by Etingof, Schedler and Soloviev~\cite{EtingofSchedlerSoloviev1999}, as well as Gateva-Ivanova and Van den Bergh~\cite{Gateva-IvanovaVandenBergh1998},
led to the study of a specific family of solutions known as \emph{involutive non-degenerate set-theoretic solutions.}

For all \( x, y \in X \), we define two maps \( f_x : X \longrightarrow X \) and \( g_y : X \longrightarrow X \) by setting \( r(x, y) = (f_x(y), g_y(x)) \). We say that the solution \((X, r)\) is \emph{involutive} if \( r^2 = \text{id}_{X \times X} \), and that \((X, r)\) is \emph{non-degenerate} if \( f_x, g_y \) are bijective maps for all \( x, y \in X \). By a \emph{solution of the YBE}, we will understand an involutive, non-degenerate set-theoretic solution of the Yang-Baxter equation.

Let \((X, r)\) be a solution of the YBE. The \emph{permutation group} of \((X, r)\) is the subgroup \(\mathcal{G}(X, r)\) of \(\text{Sym}(X)\) generated by the bijections \( f_x \) for all \( x \in X \), that is,
\[
\mathcal{G}(X, r) = \langle f_x \mid x \in X \rangle \leq \text{Sym}(X).
\]
A group \( G \) is called an \emph{involutive Yang-Baxter group}, or an IYB-group for short, if there exists a solution \((X, r)\) of the YBE such that \( G \cong \mathcal{G}(X, r) \).

To further explore such solutions, Rump~\cite{Rump2007} introduced a new algebraic structure called left brace, which was extended to a non-commutative setting by Guarnieri and Vendramin in \cite{GuarnieriVendramin2017}.

A \emph{skew (left) brace} is a triple \((B, +, \cdot)\) where \((B, +)\) and \((B, \cdot)\) are groups, referred to as the additive group (non-necessarily abelian) and the multiplicative group respectively, satisfying the following compatibility condition for all \(a, b, c \in B\):
\[
a \cdot (b + c) = a \cdot b - a + a \cdot c.
\]
Given a class of groups $\mathfrak{X}$, a skew left brace \(B\) is said to be of $\mathfrak{X}$-type if its additive group belongs to $\mathfrak{X}$. Rump's left braces are exactly the skew left braces of $\mathcal{A}$-type, where $\mathcal{A}$ is the class of all abelian groups.

Following \cite{BEJP2023}, a group $G$ is called an $\mathfrak{X}\YB$-group if $G$ is isomorphic to the multiplicative group of a skew left brace of $\mathfrak{X}$-type. 

Applying \cite[Theorem~2.1]{CedoJespersRio2010}, it follows that a group $(G, \cdot)$ is an IYB-group if and only if it can be defined a binary operation $+$ on $G$ such that $(G,+,\cdot)$ is a left brace, that is, $(G, \cdot)$ is an $\mathcal{A}\YB$-group.

If $(G,\cdot)$ is a group in $\mathfrak{X}$, then we can regard $(G,+, \cdot)$ as a \emph{trivial skew left brace} of  $\mathfrak{X}$-type with with the addition "+" defined by $a+b:=ab, a,b \in G.$ Hence $\mathfrak{X}$ is a subclass of the class of all $\mathfrak{X}\YB$-groups.


The class of all finite IYB-groups is a subclass of the class $\mathcal S$ of all finite soluble groups (\cite{EtingofSchedlerSoloviev1999}) and, by a result of Bachiller (\cite{Bachiller2016}), it is a proper subclass of $\mathcal S$. However, it includes the class of all finite nilpotent groups of class two (\cite{AultWatters1973}), the class of all finite abelian-by-cyclic groups (\cite[Corollary~3.10]{CedoJespersRio2010}), and the class of finite soluble
groups which are semidirect products $A \rtimes H$, where $A$ is nilpotent of class at most two with odd order commutator subgroup and $H$ is an IYB-group (\cite[Corollary~4.2 and Corollary~4.3]{MengBallesterEstebanFuster2021} and \cite[Proposition~2.2]{CedoOkninski2025}]).


Another important class of IYB-groups is the class of all finite soluble groups whose Sylow subgroups are abelian (\cite[Corollary~4.3]{DavidGinosar2016}). Recently, Ced\'{o} and Okni\'{n}ski proved the following significant improvement of the above result.


\begin{theorem}[{\cite[Theorem~2.6]{CedoOkninski2025}}]
Let $G$ be a finite soluble group with all Sylow subgroups of nilpotency class at most two, and Sylow $2$-subgroups abelian. Then $G$ is an $\IYB$-group.
\end{theorem}

A natural and interesting question~\cite[Question~2.7]{CedoOkninski2025} is also raised:
\begin{question}
Let $G$ be a finite soluble group such that all Sylow subgroups have nilpotency class at most two. Is $G$ an IYB-group?
\end{question}

We will address this question by establishing several significant results that contribute to the solution of this question. Our first result extends Ced\'{o} and Okni\'{n}ski's result to a wider class of all finite soluble groups with Sylow subgroups of class at most $2$.

It is well-known that the class $\mathcal{N}$ of all nilpotent groups is a class of groups which is closed under taking epimorphic images and finite subdirect products. Hence every finite group $G$ has a smallest normal subgroup with nilpotent quotient. This subgroup is called  the \emph{nilpotent residual} of $G$ and it is denoted by $G^{\mathcal{N}}$.

A group $G$ is said to be $\Q_8$-\emph{free} if $G$ has no section isomorphic to the quaternion group $Q_8$ of order $8$.


\begin{athm}\label{thm-Q-8}
Let $G$ be a finite soluble group such that all Sylow subgroups have nilpotency class at most two. If  $G^{\mathcal{N}}$ is $\Q_8$-free, then $G$ is an $\IYB$-group.
\end{athm}

In~\cite[Theorem~2.3]{CedoOkninski2025},  Ced\'{o} and Okni\'{n}ski proved that $G$ is a finite soluble Sylow tower group with all Sylow subgroups of nilpotency class at most two, then $G$ is an IYB-group provided that $G$ is either of odd order or the Sylow $2$-subgroups of $G$ are abelian. As a consequence of Theorem~\ref{thm-Q-8}, a related result holds under the weaker assumption that $G$ is $2$-nilpotent as the nilpotent residual $G^{\mathcal{N}}$ of $G$ is of odd order.

\begin{corollary}
Let $G$ be a finite soluble group with all Sylow subgroups of nilpotency class at most two. Suppose that $G$ is $2$-nilpotent. Then $G$ is an $\IYB$-group.
\end{corollary}

If $G$ is a finite soluble group with Sylow subgroups of nilpotency class at most two and Sylow $2$-subgroups isomorphic to the dihedral group of order $8$, then $G$ is an IYB-group by Theorem~\ref{thm-Q-8}. Our next result analyses the $\Q_8$-case. The reader is referred to \cite[Theorem~3.1]{Bachiller2015-p3} for the classification of left braces of order~$8$.

\begin{athm}\label{thm-Sylow-Q-8}
Let $G$ be a finite soluble group with all Sylow subgroups of nilpotency class at most two. If the Sylow $2$-subgroups of $G$ are isomorphic to $\Q_8$, then $G$ is an $\IYB$-group.
\end{athm}

Recall that, for a skew left brace $(B,+,\cdot)$, the multiplicative group $(B,\cdot)$ acts on the additive group $(B,+)$ by automorphisms: for every $a\in B$, the map $\lambda_a \colon B \rightarrow B$, given by $\lambda_a(b) = -a + ab$, is an automorphism of $(B,+)$ and the map $\lambda\colon (B,\cdot) \rightarrow \Aut(B,+)$ which sends $a \mapsto \lambda_a$ is a group homomorphism (\cite{Rump2007}). Thus, for every $a,b\in B$ it holds
\begin{equation}
\label{eq:producte-lambda} a b = a+ \lambda_a(b).
\end{equation}
We denote $\Ker(B):= \ker \lambda $, the \emph{kernel of the group action}, and 
$$\Soc(B):=\Ker(B) \cap \Z(B,+)=\{a \in B \mid ab=a+b=b+a, \forall b \in B\}$$
\emph{the socle of $B$}, which is a subbrace of $B$.

\begin{definition}[{\cite[Definition~4]{BEJP2023}}]
\begin{enumerate}
\item Let $A$ be a group and let $(B,+,\cdot)$ be a skew left brace . We say that $A$ \emph{acts on the skew left brace} $(B,+,\cdot)$ if there is an action of $A$ on the set $B$ such that $a \cdot(g+h) = a\cdot g + a\cdot h$ and $a\cdot(gh) = (a\cdot g)(a\cdot h)$, for all $g, h \in B$, that is, if the action of $A$ on the set $B$ is actually an action of $A$ on the group $(B,+)$ and an action of $A$ on the group  $(B,\cdot)$.

\item An action of a group $A$ on an $\XX\YB$-group $(G, \cdot)$, for which it is understood that the associated skew left brace of $\XX$-type  is $(G, {+}, {\cdot})$, is said to be \emph{equivariant} if $A$ acts on the skew left brace $(G,+,\cdot)$. In this case, we say that $G$ has an \emph{$A$-equivariant} $\XX\YB$-structure.
\end{enumerate}
\end{definition}

The proof of Theorem~\ref{thm-Q-8}
heavily depends on two extensions of \cite[Theorem~A]{MengBallesterEstebanFuster2021} and \cite[Theorem~A]{BEJP2023}, which are  of independent interest. They hold for arbitrary (non-necessarily finite) groups.

\begin{theorem}\label{tA}
Let $\mathfrak{X}$ be a class of groups that is closed under taking quotients and direct products, and let the group $G = NH$ be the product of subgroups $N$ and $H$ with $N \trianglelefteq G$ and $N \cap H \leq \Z(N)$. Suppose that $N$ and $H$ are both $\mathfrak{X}\YB$-groups with, respectively, associated skew left braces $(N,+, \cdot)$, $(H,+, \cdot)$ satisfying the following conditions:

\begin{enumerate}
    \item $N \cap H \leq \operatorname{Ker}(N) \cap \operatorname{Ker}(H)$.
    \item $(N \cap H, +) \leq \Z(N, +) \cap \Z(H, +)$.
    \item The action of $H$ on $N$ by conjugation in $G$ is equivariant.
\end{enumerate}
Then, $G$ is an $\mathfrak{X}\YB$-group such that $\operatorname{Soc}(N) \operatorname{C}_{\operatorname{Soc}(H)}(N) \subseteq\operatorname{Soc}(G)$. Furthermore, the associated skew left braces $(N,+, \cdot)$ and $(H,+, \cdot)$ are subbraces of the associated skew left brace $(G,+, \cdot)$.
\end{theorem}

\begin{proof}
Most part of results follows from ~\cite[Theorem~A]{BEJP2023} and we only have to show that $\Soc(N)\C_{\Soc(H)}(N) \subseteq \Soc(G)$.
Note that, in the proof of~\cite[Theorem~A]{BEJP2023}, we get that $G=NH$ is an $\mathfrak{X}\YB$-group with the addition "+" defined by
$n_1h_1+n_2h_2=(n_1+n_2)(h_1+h_2),~~~n_1,n_2 \in N, h_1,h_2 \in H$
and  $\Soc(N)\C_{\Soc(H)}(N) \subseteq \Ker(N) \C_{\Ker(H)}(N) \subseteq \Ker(G)$.  By the definition of the addition, we observe that $\Z(N,+)\Z(H,+) \subseteq \Z(G,+)$, which implies that $\Soc(N)\C_{\Soc(H)}(N) \subseteq \Z(N,+)\Z(H,+) \subseteq \Z(G,+)$, as desired.
\end{proof}




\begin{athm}\label{thm-main}\label{thm-skew-nilpotent-type}
Let $\mathfrak{X}$ be a class of groups that is closed under taking quotients and direct products, and let a group $A$ act on a group $(G, \cdot)$. Suppose that $G=A_1A_2\cdots A_k$ is a product of some $A$-invariant subgroups $A_i$,  $1 \leq i \leq k$. Write $H_0=G$, $H_i=A_{i+1}\cdots A_k$, for each $i \in \{1, \cdots,k-1\}$, and $A_{k+1} = H_k = 1$.
Suppose that $A_i$ has an $A$-equivariant $\XX\YB$-structure
$(A_i, +, \cdot)$
for each $i \in \{1, \cdots,k\}$, satisfying
the following conditions for each $i \in \{1, \cdots, k\}$:
\begin{itemize}
\item[\emph{(1)}] $A_i$ is normalised by $H_{i}$;
\item[\emph{(2)}] $(A_i, +, \cdot)$ is an $H_i$-equivariant $\XX\YB$-structure of $A_i$ with the respect to the action by conjugation of $H_i$ on $A_i$: ${}^ha=hah^{-1}$ for $h \in H_i, a \in A_i$;
\item[\emph{(3)}] $A_1A_2\cdots A_i \cap H_i \leq \C_G(A_iA_{i+1})$;
\item[\emph{(4)}] $\Z(A_i)\leq \Soc(A_i)$.
\end{itemize}
Then $G$ has an $A$-equivariant $\XX\YB$-structure $(G, +, \cdot)$ satisfying the following property:

(*) Suppose that $a=a_1a_2\cdots a_k\in \Z(A_1)\Z(A_2)\cdots \Z(A_k)$, where $a_i\in \Z(A_i)$, and $a_ia_{i+1}\cdots a_k\in \C_G(A_{i-1})$, for each $i=1, \cdots,k$. Then  $a \in \Soc(G)$.

Furthermore, $(A_i,+,\cdot)$ is a subbrace of $(G,+,\cdot)$ for every $i \in \{1, \cdots,k\}$.
\end{athm}

\begin{proof}
We argue by induction on $k$. It is clear that we may assume that  $k\geq 2$. It is rather clear that $H_1 = A_2\cdots A_{k}$ satisfies the hypotheses of the theorem. Then, by induction, $H_1$ has an $A$-equivariant $\XX\YB$-structure $(H_1,+,\cdot)$ satisfying Property~(*) and $(A_i,+,\cdot)$ are subbraces of $(H_1,+,\cdot)$ for every $i \in \{2, \cdots,k\}$.

Furthermore, $G = A_1H_1$, $H_1$ normalises $A_1$ and the $A$-equivariant $\XX\YB$-structure
$(A_1,+,\cdot )$ is $H_1$-equivariant with the respect to the action by conjugation of $H_1$ on
$A_1$ by Statement~(2). Let  $a_1=a_2\cdots a_{k} \in A_1\cap H_1$, where
$a_i \in A_i$,  $1\leq i \leq k$. Then $(a_1)^{-1}a_2\cdots a_k=1$. If $i \in \{2, \cdots, k-1\}$, then $(a_1)^{-1}a_2\cdots a_{i-1} a_i=(a_{i+1}\cdots a_{k})^{-1} \in A_1\cdots A_i\cap A_{i+1}\cdots A_{k}\leq \C_G(A_iA_{i+1})\leq \C_G(A_{i})$ and $(a_1)^{-1}a_2\cdots a_{i-1} = (a_{i}\cdots a_{k})^{-1}\in \C_G(A_{i-1}A_{i})\leq \C_G(A_{i})$ by Statement~(3). Therefore, $a_{i}\in \Z(A_{i})$. Now, $a_1\cdots a_{i-1}=(a_ia_{i+1}\cdots a_{k})^{-1}\in \C_G(A_{i-1})$ and $a_ia_{i+1}\cdots a_{k}\in \C_{G}(A_{i-1})$. Since $H_1$ satisfies Property~(*), $a_2\cdots a_k\in \Soc (H_1)$. Note that  $A_1\cap H_1\leq \C_G(A_1A_2)\cap A_1\leq \Z(A_1)\leq \Soc (A_1)$ by Statement~(4). Consequently, $A_1 \cap H_1 \leq \Soc (A_1) \cap \Soc (H_1)$. Applying Theorem~\ref{tA}, $G$ has an $A$-equivariant $\XX\YB$-structure $(G, +, \cdot)$ such that $\Soc(A_1)\C_{\Soc (H_1)}(A_1)\leq \Soc (G)$ and $(A_1, +, \cdot)$ and $(H_1,+,\cdot)$ are subbraces of $(G,+,\cdot)$. Hence $(A_i,+,\cdot)$ is a subbrace of $(G,+,\cdot)$, for every $i \in \{1, \cdots,k\}$.

Assume that $a=a_1a_2\cdots a_k\in \Z(A_1)\Z(A_2)\cdots \Z(A_k)$, where $a_i\in \Z(A_i)$, and $a_ia_{i+1}\cdots a_k\in \C_G(A_{i-1})$, for each $i=1, \cdots,k$. Then $a_2\cdots a_k\in \C_G(A_1)$ and $a_2\cdots a_k\in \Soc (H_1)$. Thus $a_1a_2\cdots a_k\in \Soc(A_1)\C_{\Soc (H_1)}(A_1)\leq \Soc (G)$. This shows that the skew left brace $(G, +, \cdot)$ has Property~(*).
\end{proof}

\emph{From now on, all groups considered will be finite.}

\section{Structural results: proof of Theorems~\ref{thm-Q-8} and~\ref{thm-Sylow-Q-8}}

This section contains the results needed for the proof of Theorem~\ref{thm-Q-8}. Our strategy is to show that a soluble group with Sylow subgroups of class at most two satisfies the structural hypotheses of Theorem~\ref{thm-main}. The elegant results developed by Hall to establish the foundations of the theory of soluble groups play a key role (\cite[Chapter~I]{DoerkHawkes1992}).

Let $G$ be a group. A \emph{Hall subgroup} of $G$ is a subgroup $H$ of $G$ such that
$(|G : H|, |H|) = 1$. If $p$ is a prime and $H$ is a Hall subgroup of $G$ such that $|G : H|$ is a power of $p$, then  $G = HP$ and $H \cap P = 1$ for every Sylow $p$-subgroup $P$ of $G$. For this reason, $H$ is sometimes called a \emph{Sylow $p$-complement} of $G$.

A celebrated result of Hall states that the existence of Sylow $p$-complements for each prime $p$ is a characteristic property of soluble groups, it is therefore not surprising that in the structural study of such groups Sylow complements plays a fundamental role (\cite[Theorems~I.3.3 and~I.3.5]{DoerkHawkes1992}).

A set $\mathcal{K}$ comprising the group $G$ together with exactly one Sylow $p$-complement of $G$ for each prime $p$ in the set $\pi(G)$ of all prime divisors of the order of $G$ is called a \emph{complemented basis} of $G$. The above result of Hall shows that $G$ is soluble if and only if $G$ has complemented basis and, in this case, $G$ has a transitive permutation representation when it acts by conjugation on the set of its complemented basis (\cite[Theorems~I.4.10 and~4.11]{DoerkHawkes1992}).

A \emph{system normaliser} of a soluble group $G$ is a stabiliser of this representation, that is, if  $\mathcal{K}$ be a complement basis of $G$, the system normaliser of $\mathcal{K}$ is  the intersection of the normalisers of every element of $\mathcal{K}$. It is denoted by $\N_G(\mathcal{K})$. It is worth noting that the system normalisers of $G$ are just the $\mathcal{N}$-normalisers of $G$ for the saturated formation $\mathcal{N}$ of all nilpotent groups (\cite[Section~3]{DoerkHawkes1992}).

According to \cite[Theorem~I.5.2]{DoerkHawkes1992}, the system normalisers of a soluble group $G$ form a characteristic conjugacy class of subgroups of $G$.

Let $G$ be a non-necessarily soluble group and let $F$ be a soluble subgroup of $G$. If $\mathcal{T}$ is a complement basis of $F$, then $\N_G(\mathcal{T}) = \{g \in G : H = H^g$ for all   $H \in \mathcal{T}\}$ is called \emph{system normaliser of $\mathcal{T}$  relative to $G$} (\cite{Hall1940}). Note that if $F$ is a normal subgroup of $G$, then $G = F\N_G(\mathcal{T})$ by the Frattini Argument.





The following lemma is a slight generalization of \cite[Theorem~I.5.8]{DoerkHawkes1992} and we give a proof here for completeness.

\begin{lemma}\label{lem-system-normaliser}
Let $G$ be a group and let $F$ be a normal soluble subgroup of $G$.
Let $\mathcal{K}=\{F_{p'} \mid p \in \pi(F)\}$ be a complement basis of $F$. Let $N$ be a normal subgroup of $G$ contained in $F$. Then $\N_G(\mathcal{K})N/N$ is the system normaliser of the complement basis $\mathcal{K}N/N=\{F_{p'}N/N \mid p \in \pi(F/N)\}$ of $F/N$ relative to $G/N$.
\end{lemma}
\begin{proof}
Without loss of generality, we may assume that $N$ is a minimal normal subgroup of $G$. Then $N$ is an elementary abelian $q$-subgroup for some prime $q$.
As $\N_G(\mathcal{K})$ normalises each $F_{p'}$ in $\mathcal{K}$, $\N_G(\mathcal{K})N/N$ normalises each $F_{p'}N/N$ in $\mathcal{K}^\ast$. Hence
$\N_G(\mathcal{K})N/N \leq \N_{G/N}(\mathcal{K}^\ast).$ Conversely, for each $gN \in \N_{G/N}(\mathcal{K}^\ast)$ with $g \in G$,
$F_{p'}^gN=F_{p'}N$ for each prime $p$.
Assume that $p = q$. Then $F_{q'}^g$ and $F_{q'}$ are two $q'$-Hall subgroups of the soluble group $F_{q'}^g N=F_{q'}N$, there exists $n \in N$ such that $F_{q'}^{gn}=F_{q'}$. Write $x=gn$. Then $gN = xN$ normalises $F_{q'}N/N$. Assume that $p \neq q$. Then $N \leq F_{p'}^g \cap  F_{p'}$ and $F_{p'}^g=F_{p'}$. Then $gN$ normalises $F_{p'}N/N$. Therefore $gN \in \N_G(\mathcal{K})N/N$.
and $\N_{G/N}(\mathcal{K}^\ast) \leq \N_G(\mathcal{K})N/N$, as desired.
\end{proof}



\begin{corollary}\label{nilres}
Let $G$ be a group and let $F$ be a normal soluble subgroup of $G$. If $D$ is a relative system normaliser of $F$ with respect to $G$, then $G = F^{\mathcal{N}}D$.
\end{corollary}

\begin{proof}
We have that $G = FD$ by the above observation. On the other hand, $H = D \cap F$ is a system normaliser of $F$ and, by Lemma~\ref{lem-system-normaliser}, $HF^{\mathcal{N}}/F^{\mathcal{N}}$ is a system normaliser of the nilpotent group $F/F^{\mathcal{N}}$. Since every $p$-complement of a nilpotent group is normal, it follows that $HF^{\mathcal{N}}/F^{\mathcal{N}} = F/F^{\mathcal{N}}$. Hence $F = F^{\mathcal{N}}H$ and so $G = F^{\mathcal{N}}D$.
\end{proof}


Suppose that a group $A$ acts on a group $G$ via automorphisms. Such action is called \emph{coprime} if $(|A|,|G|)=1$. According to \cite[8.2.6~(a)]{KurzweilStellmacher2004}, we have:
\begin{lemma}\label{lem-coprime-action}
Assume that the action of $A$ on a soluble group $G$ is coprime. For each prime $p$, there exists an $A$-invariant
$p$-complement of $G$.
\end{lemma}

\begin{lemma}\label{lem-complement}
Assume that the action of $A$ on a soluble group $G$ is coprime. Let $F$ be a normal $A$-invariant subgroup of $G$. Then there exists a complement basis $\mathcal{K}$ of $F$ such that
\begin{itemize}
\item[(1)] $\mathbf{N}_G(\mathcal{K})$ is $A$-invariant;
\item[(2)] $G=F^{\mathcal{N}}\mathbf{N}_G(\mathcal{K})$;
\item[(3)] $F^{\mathcal{N}} \cap  \mathbf{N}_G(\mathcal{K}) \leq (F^{\mathcal{N}})'$.
\end{itemize}
\end{lemma}
\begin{proof}
Since $F$ is $A$-invariant and $F$ is soluble,  it follows from Lemma~\ref{lem-coprime-action} that there exists an $A$-invariant $p$-complement $F_{p'}$ of $F$ for each $p \in \pi(F)$. Set $\mathcal{K}=\{F_{p'} \mid  p \in \pi(F)\}$. Then $\mathcal{K}$ is a complemented basis of $F$. Let $H = \mathbf{N}_G(\mathcal{K})$ be the relative normaliser of $\mathcal{K}$ relative to $G$. Clearly $H$ is $A$-invariant as $F_{p'}$ is $A$-invariant for each $p \in \pi(F)$, and Statement~(1) holds.
Statement~(2) follows from Corollary~\ref{nilres}. Write $N=F^{\mathcal{N}}$ and $D = H \cap F$. Then $D$ is a system normaliser of $\mathcal{K}$ and $DN'/N'$ is the system normaliser of
$\mathcal{K}N'/N'$ of $F/N'$ by Lemma~\ref{lem-system-normaliser}. Applying Corollary~\ref{nilres}, $F/N' = (N/N')(DN'/N')$. Since $N/N'$ is the nilpotent residual of $F/N'$ and it is abelian, we can apply \cite[Theorems~IV.5.18 and~V.4.2]{DoerkHawkes1992}, to conclude that $N/N' \cap DN'/N' = 1$, that is, $N \cap D \leq N'$. Hence $N \cap H \leq N'$ and Statement~(3) holds.
\end{proof}

The following result is the key lemma to deal with $Q_8$-free groups. It is \cite[Theorem~4]{SeitzWright1969} for the formation $\mathcal{N}$ of all nilpotent groups.
\begin{lemma}\label{lem-Q8-free}
Let $G$ be a soluble group such that $G^{\mathcal{N}}$ is $\Q_8$-free. Then $G^{\mathcal{N}}$ has an abelian Sylow $2$-subgroup.
\end{lemma}

Lemma~\ref{lem-complement} leads to an interesting characterisation of solubility. As a consequence of the well-known theorem of Kegel and Wielandt if a group $G=A_1A_2\cdots A_k$ is the product of pairwise permutable nilpotent subgroups $A_1,A_2,\cdots,A_k$, then $G$ is soluble (\cite[Corollary~2.4.4]{AFG}). We can say even more.

\begin{definition}
An \emph{$\mathcal{N}$-decomposition} of a group $G$ is a set of nilpotent subgroups $A_1,A_2,\cdots,A_k$, $k\geq 2$, of $G$ satisfiying the following properties:
\begin{enumerate}
\item $G=A_1A_2\cdots A_k$;
\item $A_i$ is normalised by $A_j$ for each $1\leq i <j\leq k$.
\item $(A_1\cdots A_i) \cap (A_{i+1} \cdots A_k) \leq A_1' \cdots A_i'$ for each $1 \leq i \leq k-1$.
\end{enumerate}
\end{definition}

\begin{theorem}\label{Ndec}
Assume that the action of a group  $A$ on a group $G$ is coprime. Then $G$ is soluble if and only if $G$ has an $\mathcal{N}$-decomposition $\{A_1,A_2,\cdots,A_k\}$  composed of $A$-invariant subgroups $A_i$, $1 \leq i \leq k$. If $G^{\mathcal{N}}$ is $\Q_8$-free, then $A_i'$ of odd order, for every $1 \leq i \leq k-1$.
\end{theorem}
\begin{proof}
If $G$ has an $\mathcal{N}$-decomposition, then $G$ is soluble by the above observation.

Assume now that the action of $A$ on the soluble group $G$ is coprime. We prove that $G$ has an $\mathcal{N}$-decomposition $\{A_1,A_2,\cdots,A_k\}$ composed of $A$-invariant subgroups (with $A_i'$ of odd order for each $1\leq i \leq k-1$ if $G^{\mathcal{N}}$ is $\Q_8$-free).

If $G$ is nilpotent, then $\{1, G \}$ is an $\mathcal{N}$-decomposition of $G$ composed of $A$-invariant subgroups. Hence we may suppose that
$G$ is not nilpotent. Then, $\F(G) \neq G$. Let $\Fit_2(G)$ be the normal subgroup of $G$ such that $\Fit_2(G)/\Fit(G) = \Fit(G/\Fit(G))$. By \cite[Theorem~A.10.6]{DoerkHawkes1992}, $\Fit_2(G) \neq \Fit(G)$. In particular,  $A_1=\Fit_2(G)^{\mathcal{N}} \neq 1$ and $A_1 \leq \Fit(G)$. Since $A_1$ is characteristic in $G$, it follows that $A_1$ is normal in $G$ and $A$-invariant.

Now  assume that $G^{\mathcal{N}}$ is $\Q_8$-free. Since $A_1 \leq G^{\mathcal{N}}$, it follows that $A_1$ is $\Q_8$-free. Applying Lemma~\ref{lem-Q8-free}, $A_1$ has an abelian Sylow $2$-subgroup. As $A_1$ is nilpotent, $A_{1}'$ is the direct product of the commutator subgroups of its all Sylow subgroups. Hence  $A_{1}'$ is of odd order.

By Lemma~\ref{lem-complement}, there exists an $A$-invariant system normaliser $H_1$ of $\Fit_2(G)$ relative to $G$ such that $G=A_1H_1$ and $A_1 \cap H_1 \leq A_1'$. Then $H_1$ is an $A$-invariant soluble proper subgroup of $G$ by Lemma~\ref{lem-complement}. If $H_1$ is nilpotent, then $\{A_1, A_2 = H_1\}$ is the desired $\mathcal{N}$-decomposition of $G$. Assume that $H_1$ is not nilpotent.

Write $H_0=G$. Assume that $i >1$ and $G$ has an $A$-invariant nilpotent subgroups $A_j$, $1 \leq j \leq i-1$, and subgroups $H_j$,  $0 \leq j \leq i-1$ of $G$ such that

\begin{enumerate}
\item $H_j = A_{j+1} \cdots A_{i-1}$;
\item $A_j=\Fit_2(H_{j-1})^{\mathcal{N}} \leq \Fit(H_{j-1})$;
\item $H_j$ is an $A$-invariant system normaliser of $\Fit_2(H_{j-1})$ relative to $H_{j-1}$;
\item $H_{j-1}=A_jH_j$  and $A_j \cap H_j \leq A_j'$
\end{enumerate}

Assume that $H_{i-1}$ is not nilpotent. Then $A_i = \Fit_2(H_{i-1})^{\mathcal{N}} \neq 1$,  is an A-invariant nilpotent normal subgroup of $H_{i-1}$. Let $H_j$ be a system normaliser of
$\Fit_2(H_{i-1})$ relative to $H_{i-1}$. Then $H_j$ is $A$-invariant, $H_{i-1}=A_iH_i$ and $A_i \cap H_i \leq A_i'$ by Lemma~\ref{lem-complement}.

Since $G$ is finite, there exists $k \geq 2$ such that $A_k = H_{k-1}$ is nilpotent. Then $G=A_1A_2\cdots A_k$ and $A_i$ is normalised by $A_j$ for each $1\leq i <j\leq k$. We prove that $(A_1\cdots A_i) \cap (A_{i+1} \cdots A_k) \leq A_1' \cdots A_i'$ for each $1 \leq i \leq k-1$ by induction on $i$. It is clear that we may assume that $i >1$ and $B_{i-1} = (A_1\cdots A_{i-1}) \cap H_{i-1}= ( A_1' \cdots A_{i-1}')  \cap H_{i-1}$.

We have that  $B_{i-1} \unlhd H_{i-1}$. Denote with bars the images in the group $\overline{H_{i-1}}=H_{i-1}/B_{i-1}$ Write $F = \Fit_2(H_{i-1})$. Since $H_i$ is a system normaliser of $F$ relative to $H_{i-1}$, it follows that $\overline{H_{i}}$ is a system normaliser of $\overline{F}$ relative to $\overline{H_{i-1}}$ by  Lemma~\ref{lem-system-normaliser}. Hence $\overline{A_{i}} \cap \overline{H_{i}} \leq \overline{A_{i}}'$ by Lemma~\ref{lem-complement}. Thus $B_{i-1}A_i \cap H_i=B_{i-1}A_i' \cap H_i$

$$
\begin{aligned}
A_1A_2\cdots A_i \cap H_i&= A_1A_2\cdots A_i \cap H_{i-1}\cap H_i\\
&=(A_1A_2 \cdots A_{i-1} \cap H_{i-1})A_i \cap H_i\\
&=B_{i-1}A_i \cap H_i=B_{i-1}A_i' \cap H_i\\
&=(A_1'A_2'\cdots A_{i-1}' \cap H_{i-1})A_i' \cap H_i\\
&=A_1'A_2'\cdots A_{i-1}'A_i' \cap H_{i-1} \cap H_i\\
&=A_1'A_2'\cdots A_{i-1}'A_i' \cap H_i.
\end{aligned}
$$
Then $\{A_1,\cdots,A_k\}$ is an $\mathcal{N}$-decomposition of $G$.

Assume that $G^{\mathcal{N}}$ is $\Q_8$-free. Let $i \in \{1, \cdots, k-1\}$. Since $A_i \leq G^{\mathcal{N}}$, it follows that $A_i$ is $\Q_8$-free. Applying Lemma~\ref{lem-Q8-free}, $A_i$ has an abelian Sylow $2$-subgroup. Since $A_i$ is nilpotent, $A_{i}'$ is the direct product of the commutator subgroups of its Sylow subgroups. Hence $A_{i}'$ is of odd order.
\end{proof}

We shall make repeated use of the following two elementary lemmas.

\begin{lemma}\label{lem-class-two-property}
Let $G$ be a group with all Sylow subgroups of $G$ of nilpotency class at most two. Suppose that $N$ is a nilpotent subgroup of $G$. Then
\begin{itemize}
\item[\emph{(1)}] $N$ has nilpotency class at most two;
\item[\emph{(2)}] If $N$ is normalised by a subgroup $H$ of $G$, then $N' \cap H \leq\Z(\Fit(H))$.
\end{itemize}
\end{lemma}
\begin{proof}
As $N$ is nilpotent, $N=P_1 \times \cdots \times P_s$, where  $\pi(N)=\{p_1,\cdots,p_s\}$ and $P_i$ be a Sylow $p_i$-subgroup of $N$, $1 \leq i \leq s$. Since $P_i$ has nilpotent class at most $2$ for all $i$, we have that $N$ has nilpotency class at most two by \cite[Theorem~A.8.2]{DoerkHawkes1992} and Statement~$(1)$ holds.

Assume that $N$ is normalised by $H$. Then $N \cap H$ is a normal nilpotent subgroup of $H$ and so $N \cap H$ is contained in $\Fit(H)$, the Fitting subgroup of $H$. Write $N' \cap H=Q_1 \times \cdots \times Q_s$, where $Q_i$ is a Sylow $p_i$-subgroup of $N' \cap H$, $1 \leq i \leq s$.

Fix an index $i \in \{1, \cdots, s\}$. Since $N'=P_1' \times \cdots \times P_s'$, it follows that $Q_i \leq P_i'$ as $P_i'$ is the unique Sylow $p_i$-subgroup of $N'$.  Let $R_i$ be the Sylow $p_i$-subgroup of $\Fit(H)$. Then $Q_i \leq R_i$. Since $H$ normalises $N$, we have that $R_i$ normalises $P_i$. Hence $R_iP_i$ is a $p_i$-subgroup of $G$. Since $R_iP_i$ has nilpotency class at most two, it follows that $[Q_i,R_i]\leq [P_i', R_i]=1,$. Thus $Q_i \leq \Z(R_i) \leq \Z(\Fit(H))$. Consequently, $N' \cap H\leq \Z(\Fit(H))$ and Statement~(2) holds.
\end{proof}


\begin{lemma}\label{lem-Sylow}
Let a group $G=A_1A_2\cdots A_k$ be the product of some nilpotent subgroups $A_1,A_2,\cdots,A_k$ of $G$ such that $A_i$ is normalised by $A_j$ for $1 \leq i<j \leq k$. Let $p$ be a prime and let $P_i$ be the Sylow $p$-subgroup of $A_i$ for each $i$. Then $P_1P_2\cdots P_k$ is a Sylow $p$-subgroup of $G$.
\end{lemma}
\begin{proof}
We argue by induction on $k$. We can assume that $k > 1$  and the result is true for $k-1$. Note that $A=A_1A_2\cdots A_{k-1}$ is a subgroup of $G$ satisfying the hypotheses of the lemma. Hence $P=P_1P_2\cdots P_{k-1}$ is a Sylow $p$-subgroup of $A=A_1A_2\cdots A_{k-1}$. Since $A_i$ is nilpotent, $P_i$ is a characteristic subgroup of $A_i$. By hypothesis, $A_k$ normalises $A_i$ for $1 \leq i \leq k-1$. Hence $P_k \leq A_k$ normalises $P_i$ for $1 \leq i \leq k-1$. Consequently $P_k$ normalises $P=P_1P_2\cdots P_{k-1}$ and $PP_k$ is a $p$-subgroup of $G$. Let $|G|_p$be denote the largest power of $p$ dividing the order of $G$. Since $P \cap P_k$ is a $p$-subgroup of $A \cap A_k$, it follows that
$$|G|_p=\frac{|A|_p\cdot |A_k|_p}{|A \cap A_k|_p}~\text{divides}~\frac{|P|\cdot |P_k|}{|P \cap P_k|}=|PP_k|.$$
Hence $PP_k=P_1P_2\cdots P_k$ is a Sylow $p$-subgroup of $G$, as desired.
\end{proof}

The following lemmas turn out to be crucial to prove Theorem~\ref{thm-Q-8}.

\begin{lemma}\label{lem-nilpotency-class-2}
Assume that the action of a group  $A$ on a soluble group $G$ is coprime and all Sylow subgroups of $G$ have nilpotency class at most two. Let $\{A_1,A_2,\cdots,A_k\}$ be an $\mathcal{N}$-decomposition of $G$ composed of $A$-invariant subgroups. Then
$$A_1A_2\cdots A_i \cap A_{i+1}\cdots A_k \leq \C_G(A_iA_{i+1}) \text { for each } 1 \leq i \leq k-1.$$
\end{lemma}
\begin{proof}

Write $H_i=A_{i+1}A_{i+2} \cdots A_k$ for $1 \leq i \leq k-1$. For each prime $p$, let $A_{i,p}$  be the Sylow $p$-subgroup of $A_i$, $1 \leq i \leq k$. By Lemma~\ref{lem-Sylow},
$G_p = A_{1,p}A_{2,p} \cdots A_{k,p}$ is a Sylow $p$-subgroup of $G$, and $A_{j,p}' \leq G_p'$ for each $j \in \{1, \cdots, k\}$.

Let $1 \leq i \leq k-1$ and write $B_i=A_1A_2\cdots A_i \cap H_i$. Let $B_{i,p}$ be a Sylow $p$-subgroup of $B_i$, Clearly $B_i \unlhd H_i$ as $H_i$ normalises $A_1A_2\cdots A_i$.

We prove the following statements by induction on $i$:
\begin{itemize}
\item[($\ast$-1)] $B_i=A_1'A_2' \cdots A_i' \cap H_i \leq \Z(\Fit(H_i)) \cap \C_G(A_{i})$;
\item[($\ast$-2)] $B_{i,p} \leq G_p'$ for each prime $p$.
\end{itemize}
If $i=1$, then
$B_1=A_1 \cap H_1 =A_1' \cap H_1$. Since $A_1$ is nilpotent and $A_{1,p}$ is the Sylow $p$-subgroup of $A_1$, we have that $A_{1,p}'$ is the Sylow $p$-subgroup of $A_1'$. Hence $B_{1,p} \leq A_{1,p}' \leq G_p'$ and ($\ast$-2) follows. Moreover, as $A_1$ is a nilpotent normal subgroup of $G$, it follows from Lemma~\ref{lem-class-two-property} that
$A_1' \cap H_1 \leq \Z(\Fit(H_1))$. Moreover, $A_1' \leq \Z(A_1)$ as $A_1$ is of nilpotency class $2$. Hence ($\ast$-1) and ($\ast$-2) holds for $i=1$.

Now we assume that $i > 1$ and ($\ast$-1) and ($\ast$-2) holds for $i-1$.
Then  $B_{i-1, p} \leq G_p'$ for each prime $p$ and
$B_{i-1}=A_1'A_2' \cdots A_{i-1}' \cap H_{i-1} \leq \Z(\Fit(H_{i-1})) \cap \C_G(A_{i-1}).$


Since $B_{i-1}$ centralizes $A_i \leq \Fit(H_{i-1})$ and $A_i' \leq \Z(A_i)$ by hypothesis,
$B_{i-1}A_i'$ centralizes $A_i$ and so $B_i=B_{i-1}A_i' \cap H_i \leq \C_G(A_i)$.
Moreover, $B_i \leq B_{i-1}A_i'$ is nilpotent.
It follows from Lemma~\ref{lem-Sylow} that
$B_{i-1,p}A_{i,p}'$ is the unique Sylow $p$-subgroup of $B_{i-1}A_i'$ and so $B_{i,p} \leq B_{i-1,p}A_{i,p}'$. By induction, $B_{i-1,p} \leq G_p'$. Hence $B_{i,p} \leq B_{i-1,p}A_{i,p}' \leq G_p'$ and ($\ast$-2) holds for the case $i$.

Since $B_{i} \unlhd H_i$ and $B_i$ is nilpotent, we get $B_i \leq \Fit(H_i)$, moreover, $B_{i,p} \leq \Oo_p(H_i)$.
As $H_i=A_{i+1} \cdots A_{k}$, by Lemma~\ref{lem-Sylow}, $A_{i+1,p}A_{i+2,p} \cdots A_{k,p}$ is a Sylow $p$-subgroup of $H_i$. Clearly $\Oo_p(H_i) \leq A_{i+1,p}A_{i+2,p} \cdots A_{k,p} \leq G_p$.
Since $G_p$ is nilpotent of class two by hypothesis, $[B_{i,p}, \Oo_p(H_i)] \leq [G_p',G_p]=1$. Hence $B_{i,p} \leq \Z(\Oo_p(H_i))$, which implies that $B_i \leq \Z(\Fit(H_i))$ and ($\ast$-1) holds.

Since $A_{i+1} \leq \Fit(H_i)$, we get that
$A_1A_2\cdots A_i \cap A_{i+1}\cdots A_k \leq \C_G(A_iA_{i+1})$. The proof of the lemma is now complete.
\end{proof}



The following known result is a consequence of Theorem~\ref{thm-main}.

\begin{lemma}[{\cite[Lemma~16]{CedoJespersRio2010}}]\label{nilp}
Every nilpotent group of nilpotency class at most two is an IYB-group.
\end{lemma}

\begin{proof}
Let $G$ be a nilpotent group of nilpotency class at most $2$. Then $G = A_1A_2\cdots A_k$ is the direct product of $\{ A_1, \cdots,A_k\}$  where $A_i$ is a Sylow $p_i$-subgroup of $G$, $1\leq i \leq k$, $\pi(G) = \{p_1,\cdots, p_k\}$. Assume that $k = 1$. Then $G$ is a $p$-group for some $p$, and $G$ is either cyclic or $G$ has an $\mathcal{N}$-decomposition of $G$  composed of abelian groups since $G/Z(G)$ is a direct product of cyclic subgroups. Since every abelian group is an IYB-group, it follows that $G$ is an IYB-group by Theorem~\ref{thm-main}.

Assume that $k \geq 2$. Then $\{ A_1, \cdots,A_k\}$ is an $\mathcal{N}$-decomposition of $G$ composed of IYB-groups. By Lemma~\ref{lem-nilpotency-class-2}, $G$ satisfies the hypotheses of Theorem~\ref{thm-main} for $A = 1$. Hence $G$ is an IYB-group.

\end{proof}

Applying~\cite[Corollary~4.2]{MengBallesterEstebanFuster2021}, we have:

\begin{lemma}\label{lem-odd-order-commutator}
Let $G$ be a nilpotent group of nilpotency class at most two with $|G'|$ odd. Then $G$ has an IYB-structure which is \(\Aut(G)\)-equivariant with respect to the natural action of \(\text{Aut}(G)\) on \( G \) and $\Z(G) \subseteq \Ker(G)$.
\end{lemma}

\begin{proof}[\textbf{\emph{Proof of Theorem~\ref{thm-Q-8}}}]
By Theorem~\ref{Ndec}, Lemmas~\ref{lem-nilpotency-class-2},~\ref{nilp} and~\ref{lem-odd-order-commutator}, $G$ has an $\mathcal{N}$-decomposition satisfying the statements of Theorem~\ref{thm-main} for $A = 1$. Hence $G$ is an IYB-group.
\end{proof}

In order to prove Theorem~\ref{thm-Sylow-Q-8}, we need the following IYB-structure on $\Q_8$.
\begin{lemma}\label{lem-IYB-Q-8}
Let $G=\Q_8$ and $A=\Aut(G)$. Then $G$ has an $\Oo^2(A)$-equivariant IYB-structure $(G,+,\cdot)$ with $\Z(G) \leq \Ker(G)$.
\end{lemma}
\begin{proof}
Let $G=\Q_8=\langle a,b \mid a^4=b^4=1, a^2=b^2, [a,b]=b^2\rangle$. Write $c=a^2=b^2 \in \Z(G)$.  Note that for every element $u \in G$, there exists the unique triple $(x,y,z) \in \mathbb{F}_2\oplus \mathbb{F}_2\oplus \mathbb{F}_2 $ such that $u=a^xb^yc^z$, where $\mathbb{F}_2=\{0,1\}$ is a field of order $2$. We define an addition on $G$ as
$$a^{x_1}b^{y_1}c^{z_1}+a^{x_2}b^{y_2}c^{z_2}= a^{x_1+x_2}b^{y_1+y_2}c^{z_1+z_2}.$$
Note that $1+1=0$ in the field $\mathbb{F}_2$. Hence $a^{x_1}a^{x_2}=a^{x_1+x_2}c^{x_1x_2}, b^{x_1}b^{x_2}=b^{x_1+x_2}c^{x_1x_2}$ and, more generally,
$$(a^{x_1}b^{y_1}c^{z_1})(a^{x_2}b^{y_2}c^{z_2})=a^{x_1+x_2} b^{y_1+y_2} c^{z_1+z_2+x_2y_1+y_1y_2+x_1x_2}$$
Then $(G,+,\cdot)$ is a left brace with $(G,+)$ elementary Abelian.
Note that $\lambda_c(u)=cu-c=a^xb^yc^{z+1}-c=a^xb^yc^z=u$ for each $u=a^xb^yc^z \in G$. Hence $\Z(G)=\langle c\rangle \leq \Ker(G)$.

Now we will show that this IYB-structure on $\Q_8$ is $\Oo^2(A)$-equivariant, that is,  $\Oo^2(A) \leq \Aut(G,+,\cdot)$. Since $A \cong S_4$, it suffices to prove that every element of $A$ with order $3$ belongs to $\Aut(G,+,\cdot)$. Note that every automorphism with order $3$ of $\Q_8$ permutates its three cyclic subgroups of order $4$. Hence we can easily list its all automorphism with order $3$:
$$
\begin{aligned}
&\alpha_1: a \mapsto b; b \mapsto ab;~~~~~~~~~~~\beta_1: a \mapsto ab; b \mapsto a;\\
&\alpha_2: a \mapsto b; b \mapsto (ab)c;~~~~~~\beta_2: a \mapsto (ab)c; b \mapsto a;\\
&\alpha_3: a \mapsto bc; b \mapsto ab;~~~~~~~~~\beta_3: a \mapsto (ab)c; b \mapsto ac;\\
&\alpha_4: a \mapsto bc; b \mapsto (ab)c;~~~~\beta_4: a \mapsto ab; b \mapsto ac;\\
\end{aligned}$$
Note that $\beta_i=\alpha_i^2$ for each $1 \leq i \leq 4$ and $\Oo^2(A)=\langle \alpha_i \mid 1 \leq i \leq 4 \rangle$.
Note that $(ab)^y=a^yb^y$ for $y=0,1$. We can check that $\alpha_i \in \Aut(G,+)$ for each $1 \leq i \leq 4$.

Since $\alpha_1(a^xb^yc^z)=b^x(ab)^yc^z=b^xa^yb^yc^z=a^yb^x[b^x,a^y]b^yc^z
=a^yb^xb^yc^{z+xy}=a^yb^{x+y}c^{xy}c^{z+xy}=a^yb^{x+y}c^z$ for $x,y,z \in \mathbb{F}_2$, it follows that
$$
\begin{aligned}
\alpha_1(a^{x_1}b^{y_1}c^{z_1}+a^{x_2}b^{y_2}c^{z_2})&=\alpha_1(a^{x_1+x_2}b^{y_1+y_2}c^{z_1+z_2})\\
&=a^{y_1+y_2}b^{x_1+x_2+y_1+y_2}c^{z_1+z_2}\\
&=a^{y_1}b^{x_1+y_1}c^{z_1}+a^{y_2}b^{x_2+y_2}c^{z_2}\\
&=\alpha_1(a^{x_1}b^{y_1}c^{z_1})+\alpha_1(a^{x_2}b^{y_2}c^{z_2}).
\end{aligned}
$$
Hence $\alpha_1 \in \Aut(G,+,\cdot)$. Similarly, for $x,y,z \in \mathbb{F}_2$, we have

$$
\begin{aligned}
\alpha_2(a^xb^yc^z)&=b^x(abc)^yc^z=a^yb^{x+y}c^{z+y}\\
\alpha_3(a^xb^yc^z)&=(bc)^x(ab)^yc^z=a^yb^{x+y}c^{z+x}\\
\alpha_4(a^xb^yc^z)&=(bc)^x(abc)^yc^z=a^yb^{x+y}c^{z+x+y}\\
\end{aligned}
$$
It can be directly verified that $\alpha_2,\alpha_3,\alpha_4 \in \Aut(G,+,\cdot)$ as well. Hence such IYB-structure on $\Q_8$ is $\Oo^2(A)$-equivariant and the proof is complete.
\end{proof}

Now we can prove Theorem~\ref{thm-Sylow-Q-8}.

\begin{proof}[\textbf{\emph{Proof of Theorem~\ref{thm-Sylow-Q-8}}}]
Since $G$ is soluble and all Sylow subgroups of $G$ have nilpotency class at most two, by Theorem~\ref{Ndec} and Lemma~\ref{lem-nilpotency-class-2}, $G$ has an $\mathcal{N}$-decomposition $\{A_1,A_2,\cdots,A_k\}$ such that
$A_1A_2\cdots A_i \cap A_{i+1}\cdots A_k \leq \C_G(A_iA_{i+1})$ for each $1\leq i \leq k-1$.
Write $H_i=A_{i+1}A_{i+2} \cdots A_k$ for each $1 \leq i \leq k-1$ and $H_k=1$.
We will apply Theorem~\ref{thm-main} for $A=1$. Hence we only have to show that $A_i$ has an $H_i$-equivariant IYB-structure $(A_i,+, \cdot)$ and $\Z(A_i) \leq \Ker(A_i)$ for each $i$.

Let $P_i$ be a Sylow $2$-subgroup of $A_i$ for each $i$ and write $P=P_1P_2\cdots P_k$, which is a
Sylow $2$-subgroup of $G$ by Lemma~\ref{lem-Sylow}.  Fixed some $i \in \{1,2,\cdots,k\}$. If $|P_i| \leq 2^2$, it implies that $P_i$ is Abelian. As $A_i$ is nilpotent, $A_i'$ is of odd order. It follows from Lemma~\ref{lem-odd-order-commutator} that $A_i$ has an $\Aut(A_i)$-equivariant IYB-structure $(A_i,+, \cdot)$ and $\Z(A_i) \leq \Ker(A_i)$, as desired.

Now assume that $|P_i| \geq 2^3$. It implies that $P_i=P$ as $|P|=2^3$ by hypothesis. Write $A_i=P \times W$, where $W$ is the Hall $2'$-subgroup of $A_i$. Considering the action of $H_i$ on $A_i=P \times W $, for each $j>i$, $P_j=P_j \cap P=P_j \cap P_i \leq A_i \cap H_j \leq \C_G(A_i)$. Hence $P_j \leq \Z(P)$, which implies that $P_{i+1}P_{i+2} \cdots P_k \leq \C_{H_i}(P)$. Note that $H_i/\C_{H_i}(P)$ is of order odd and acts faithfully on $P$. It follows from Lemma~\ref{lem-IYB-Q-8} that $P$ has an IYB-structure $(P,+,\cdot)$, which is $H_i/\C_{H_i}(P)$-equivariant, also $H_i$-equivariant, and $\Z(P) \leq \Ker(P)$.
Since $W$ is nilpotent of class two with odd order, by Lemma~\ref{lem-odd-order-commutator}, $W$ has an $H_i$-equivariant IYB-structure $(W,+,\cdot)$ with $\Z(W) \leq \Ker(W)$. It follows from Theorem~\ref{tA} that
$A_i$ has an $H_i$-equivariant IYB-structure $(A_i,+,\cdot)$ such that $\Z(A_i)=\Z(P)\Z(W) \leq \Ker(P)\Ker(W)=\Ker(A_i)$, as desired.
\end{proof}

\section{Some corollaries: skew braces of $\mathcal{N}$-type.}

In this section we show that \cite[Theorem~3.2,Theorem~3.5 ]{CedoOkninski2025} are natural consequences of our results.

Let $(B,+,\cdot)$ a finite skew left brace of $\mathcal{N}$-type. Then $(B,+)$ is nilpotent and so it has a $\lambda$-invariant Hall $\pi$-subgroup for all set of primes $\pi$. In particular, $(B,\cdot)$ has Hall $\pi$-subgroups for all set of primes $\pi$. By \cite[Theorems~I.3.3 and~I.3.5]{DoerkHawkes1992}), $(B,\cdot)$ is soluble. Hence every finite $\mathcal{N}\YB$-group is soluble (see~\cite{Byott2015}).

\begin{corollary}
Let $G$ be a group with the Sylow tower property. Then $G$ is an $\mathcal{N}\YB$-group.
\end{corollary}
\begin{proof}
Since $G$ is a group with the Sylow tower property, we may assume that $G=A_1A_2\cdots A_k$ with $\pi(G)=\{p_1,p_2,\cdots,p_k\}$, where $A_i$ is a Sylow $p_i$-subgroup of $G$ such that $A_1A_2\cdots A_i \unlhd G$ and $A_1A_2\cdots A_i \cap A_{i+1}\cdots A_k=1$. Since $A_i$ is nilpotent, we can get $(A_i,+,\cdot)$ the trivial skew brace of nilpotent type, which implies that $A_i$ has fully equivariant $\mathcal{N}\YB$-structure $(A_i,+,\cdot)$ with $\Ker(A_i)=A_i$, moreover, $\Z(A_i)=\Z(A_i,+)=\Z(A_i,+) \cap \Ker(A_i)=\Soc(A_i)$.
Applying Theorem~\ref{thm-skew-nilpotent-type}, $G$ is an $\mathcal{N}\YB$-group.
\end{proof}

\begin{corollary}
Let $G$ be a soluble group with Sylow subgroups
of nilpotency class at most $2$. Then $G$ is a $\mathcal{N}\YB$-group.
\end{corollary}
\begin{proof}
We can apply Theorem~\ref{Ndec} and Lemma~\ref{lem-nilpotency-class-2} to conclude that 
$G$ has an $\mathcal{N}$-decomposition $\{A_1,\cdots,A_k\}$ such that
$A_1A_2\cdots A_i \cap A_{i+1}\cdots A_k \leq \C_G(A_iA_{i+1})$ for each $i$. Every nilpotent group $A_i$ has a trivial skew brace $(A_i,+,\cdot)$ with $\Z(A_i)=\Soc(A_i)$. By Theorem~\ref{thm-skew-nilpotent-type}, $G$ is an $\mathcal{N}\YB$-group.
\end{proof}

\textbf{Acknowledgment:} The first author is sponsored by Natural Science Foundation of Shanghai
(24ZR1422800) and National Natural Science Foundation of China (12471018). The second author is supported by the grant CIAICO/2023/007 from the Conselleria d'Educaci\'o, Universitats i Ocupaci\'o, Generalitat Valenciana and by the program "High-end Forcing Expert Program of China(H20240848)", Shanghai University.

\bibliographystyle{plain}
\bibliography{bibM}
\end{document}